\documentclass[12pt,reqno]{amsart}
\usepackage{amssymb}
\usepackage{amsthm}
\usepackage{latexsym,amsmath,amssymb}
\usepackage{ascmac} 

\usepackage{braket}
\usepackage{amscd}
\usepackage{amsfonts}
\usepackage{mathrsfs}
\usepackage{bm}  
\usepackage{enumitem}  
\usepackage{hhline} 
\usepackage{appendix}
\usepackage[colorlinks=true, bookmarks=true,
bookmarksnumbered=true, bookmarkstype=toc, linkcolor=blue,
urlcolor=blue, citecolor=blue]{hyperref}

\def\Xint#1{\mathchoice
{\XXint\displaystyle\textstyle{#1}}%
{\XXint\textstyle\scriptstyle{#1}}%
{\XXint\scriptstyle\scriptscriptstyle{#1}}%
{\XXint\scriptscriptstyle\scriptscriptstyle{#1}}%
\!\int}
\def\XXint#1#2#3{{\setbox0=\hbox{$#1{#2#3}{\int}$}
\vcenter{\hbox{$#2#3$}}\kern-.5\wd0}}

\def\dashint{\Xint-}

\theoremstyle{plain} 
\newtheorem{thm}{Theorem}[section]
\newtheorem{lem}[thm]{Lemma}
\newtheorem{coro}[thm]{Corollary}
\newtheorem{prop}[thm]{Proposition}

\theoremstyle{definition} 
\newtheorem{dfn}[thm]{Definition}
\newtheorem{rem}[thm]{Remark}
\makeatletter
    
    \@addtoreset{equation}{section}
  \makeatother

\makeatletter
\def\tr{\mathop{\operator@font tr}\nolimits}  
\makeatother

\makeatletter
\def\dist{\mathop{\operator@font dist}\nolimits}  
\makeatother

\makeatletter
\def\div{\mathop{\operator@font div}\nolimits}  
\makeatother

\makeatletter
\def\exp{\mathop{\operator@font exp}\nolimits}  
\makeatother

\makeatletter
\def\essinf{\mathop{\operator@font {\itshape ess}.\inf}\nolimits}  
\makeatother

\makeatletter
\def\esssup{\mathop{\operator@font {\itshape ess.}\sup}\nolimits}  
\makeatother

\newcommand{\R}{\mathbb{R}}

\newcommand{\M}{{\mathcal M}}
\newcommand{\A}{{\mathcal A}}

\def\a{\alpha}

\def\phi{\varphi}
\def\e{\varepsilon}

\def\ol{\overline}

\newcommand{\Tr}{\mathop{\mathrm{Tr}}\nolimits} 

\begin{document}

\title[$W^{\sigma,p}$ A Priori Estimates]{$ W^{\sigma,p}$ A Priori Estimates for Fully Nonlinear Integro-Differential Equations}

\author[S. KITANO]{
Shuhei Kitano
}


\address{
Department of Applied Physics\endgraf
Waseda University\endgraf
Tokyo, 169-8555\endgraf
JAPAN
}
\email{sk.koryo@moegi.waseda.jp}

\maketitle

\begin{abstract}
$W^{\sigma,p}$ estimates are
studied for a class of
fully nonlinear integro-differential equations of order $\sigma$, which are analogues of $W^{2,p}$ estimates by Caffarelli. We also present Aleksandrov-Bakelman-Pucci maximum principles, which are improvements of estimates proved by Guillen-Schwab, depending only on $L^p$ norms of inhomogeneous terms.
\end{abstract}

\section{Introduction}
\label{sec:intro}

In this paper, we study the fully nonlinear nonlocal equation of the form:
\begin{equation}\label{eq1}
Iu(x):=\inf_{\beta\in\mathscr{B}}\frac{\A(n,-\sigma)}{2}\int_{\R^n}\delta(u,x,y)\frac{y^T A_\beta(x) y}{|y|^{n+2+\sigma}}dy=f(x)\quad\mbox{in }B_3,
\end{equation}
where $n\geq2$, $\sigma\in(0,2)$, $\mathscr{B}$ is an index set, $\A(n,-\sigma)$ is the normalizing constant, defined by \eqref{-A} in Section \ref{S2}, and we set
\[
\delta(u,x,y):=u(x+y)+u(x-y)-2u(x).
\]
Here, $B_3$ is the open ball with radius $3$ and center $0$ and for each $\beta\in\mathscr{B},x\in\R^n$, $A_\beta(x)$ is the the $n\times n$ real symmetric matrix.
We assume the ellipticity condition by following \cite{GS12}: there exist $\lambda,\Lambda>0$ such that  $A_\beta(x)\geq O$ and
\[
\lambda Id\leq \frac{1}{n+\sigma}\{\sigma A_\beta(x)+\Tr(A_\beta(x))Id\}\leq\Lambda Id\quad\mbox{for }\beta\in\mathscr{B},x\in B_3.
\]

The main purpose of this paper is to present the $W^{\sigma,p}$ estimate for the nonlocal equation \eqref{eq1}, which is an analogue of the $W^{2,p}$ estimate for fully nonlinear second order equations.
We also show the Aleksandrov-Bakelman-Pucci (ABP for short) maximum principle, which provides the bound of supersolutions of a nonlocal equation by the $L^p$ norm of the right hand side.

On the other hand, in~\cite{Caf}, Caffarelli establish several interior a priori estimates for fully nonlinear second order equations.
These are generalizations of classical results for linear second order equations and proved by a new approach based on the ABP maximum principle.
The $W^{2,p}$ estimate is one of main results in~\cite{Caf}, where the case of $p>n$ is treated.
Escauriaza~\cite{Esc} extends $W^{2,p}$~estimates to the range $p>p^*$, where $p^*\in[n/2,n)$ depends only on ellipticity constants.

Initiated by researches from stochastic analysis such as \cite{BK05,BL02}, Silvestre \cite{Sil06} started H\"older estimates for integro-differential equations.
Afterwards, Caffarelli-Silvestre investigated the regularity theory of viscosity solutions for full nonlinear nonlocal equations in~\cite{CS09,CS11a,CS11b}.
In \cite{CS09}, they introduce a class of elliptic integro-differential equations, which includes the following equation:
\begin{equation}\label{eq2}
\inf_{\beta}\sup_{\alpha}\int_{\R^n}\delta(u,x,y)K_{\alpha,\beta}(x,y)dy=f(x),
\end{equation}
where $K_{\alpha,\beta}$ satisfies the ellipticity condition:
\[
\frac{(2-\sigma)\lambda}{|y|^{n+\sigma}}\leq K_{\alpha,\beta}(x,y)\leq\frac{(2-\sigma)\Lambda}{|y|^{n+\sigma}}.
\]
Then, Harnack inequalities and H\"older estimates are obtained by applying nonlocal versions of the ABP maximum principle,
and $C^{1,\alpha}$~estimates are also proved for some restrictive classes of translation-invariant equations.
Note that constants obtained in these results do not blow up as the order $\sigma$ goes to 2.
In~\cite{CS11b}, they discussed variable coefficient elliptic equations and established Cordes-Nirenberg type estimates.
An analogue of the Evans-Krylov theorem is provided for a class of translation-invariant convex nonlocal equations in~\cite{CS11a}.
However, some smoothness of the kernel $K_{\alpha,\beta}$ are required for $C^{1,\alpha}$ estimates and the nonlocal Evans-Krylov theorem in~\cite{CS09,CS11a,CS11b}. 
In the case where we do not assume any regularities of the kernel with respect to $y$, Kriventsov~\cite{Kriv} improves $C^{1,\alpha}$~estimates, and Serra~\cite{Serra} establishes a nonlocal Evans-Krylov theorem and $C^{\sigma+\alpha}$~Schauder type estimates.

A natural question is whether $W^{\sigma,p}$~estimates, which are analogous to $W^{2,p}$~estimates by~\cite{Caf,Esc}, are valid for the nonlocal equation \eqref{eq2}.
However, this question is still largely open and one of difficulties comes from that nonlocal ABP maximum principles in~\cite{CS09}, strongly depend on $L^\infty$ norms of inhomogeneous terms $f$.
Another versions of ABP maximum principles are established by Guillen-Schwab~\cite{GS12}.
They study supersolutions of the fractional Pucci equation:
\[
\left\{
\begin{aligned}
\sup_{\lambda\leq\Tr(A)\mbox{ and }A\leq \Lambda Id}(2-\sigma)\int_{\R^n}\delta(u,x,y)\frac{y^TA y}{|y|^{n+2+\sigma}}dy&=f(x)&\quad&\mbox{in }B_1,\\
u(x)&=0&\quad&\mbox{in }\R^n\setminus B_1.
\end{aligned}
\right.
\]
This equation is more restrictive in than the one treated in~\cite{CS09}, however the ABP maximum principle in \cite{GS12} is represented as follow:
\[
-\inf_{B_1} u\leq C\|f\|_\infty^{(2-\sigma)/2}\|f\|_n^{\sigma/2},
\]
and plays a fundamental role of the proof of $W^{\sigma,\epsilon}$~estimates by Yu~\cite{Y}, which are inspired by $W^{2,\e}$-estimates of  \cite{E,L}.
Recently, the author proves ABP maximum principles depending only on $\|f\|_n$, where $\sigma$ is sufficiently close to $2$ in~\cite{K}.

This paper is organized as follows.
In Section \ref{S2} we recall several preliminary results of viscosity solutions of second order equations and nonlocal equations.
In Section \ref{S3}, we will show ABP maximum principles depending only on $L^p$ norms of inhomogeneous terms.
In Section \ref{S4}, we derive $W^{\sigma,p}$~estimates.
Finally, in Section \ref{S5}, we construct some functions, to show ABP~type estimates do not hold for some $p>n/\sigma$.
In the Appendix, we prove basic properties of inf-convolutions of viscosity supersolutions for nonlocal equations.


\section{Preliminaries}\label{S2}

Throughout this paper, we let   $|\cdot|$ be the Euclidean norm. 
We set
\[
B_r(x):=\{y\in\R^n:|x-y|<r\}
\]
and $B_r:=B_r(0)$. 
We write $u^+:=\max\{u,0\}$ and $u^-:=\max\{-u,0\}$.
For $\Omega\subset\R^n$ open, we denote as $C_0(\Omega)$ the space of continuous compactly supported functions in $\Omega$
and for given $\gamma>0$, also denote as $C^{\gamma}(\Omega)$ the space $C^{k,\gamma-k}(\Omega)$, where $k\in \mathbb{Z}$ is the floor of $\gamma$.

We often use the following normalizing constants: for $n\geq2,\sigma\in(0,2)$,
\begin{align}\label{A}
\A(n,2-\sigma)&:=\frac{\Gamma((n+\sigma-2)/2)}{\pi^{n/2}2^{2-\sigma}\Gamma((2-\sigma)/2)}\quad\mbox{and}\\
\A(n,-\sigma)&:=\frac{2^\sigma\Gamma((n+\sigma)/2)}{\pi^{n/2}|\Gamma(-\sigma/2)|},\label{-A}
\end{align}
where $\Gamma(-s)=\Gamma(1-s)/(-s)$ is the Gamma function evaluated at $-s$.
Note that $\A(n,-\sigma)/(\sigma(2-\sigma))$ remains between two positive constants for $\sigma\in(0,2)$.

Let us define the $\sigma$ order fractional Hessian $D^\sigma u$ by
\[
D^\sigma u(x):=\frac{\A(n,-\sigma)}{2}\int_{\R^n}\delta(u,x,y)\frac{y\otimes y}{|y|^{n+2+\sigma}}dy,
\]
where $y\otimes y$ is the real symmetric $n\times n$ matrix defined by $(y\otimes y)_{i,j}:=y_iy_j$ for $1\leq i,j\leq n$.
Then, our nonlocal operator $I$ in \eqref{eq1} can be written by
\begin{equation}\label{operator}
Iu(x)=\inf_{\beta\in\mathscr{B}}\Tr (A_\beta(x)D^\sigma u(x)).
\end{equation}
Let $S_{\lambda,\Lambda}$ be the set of real symmetric $n\times n$ matrices $A$ such that
\[
\lambda Id\leq A_\sigma\leq\Lambda Id\quad \mbox{and}\quad A\geq O,
\]
where
\begin{equation}\label{As}
A_\sigma:=\frac{1}{n+\sigma}\{\sigma A+\Tr(A)Id\}.
\end{equation}
The fractional Pucci maximal and minimal operator is defined by
\begin{equation*}
\M^+u(x):=\sup_{A\in S_{\lambda,\Lambda}}\Tr(AD^\sigma u(x))=\sup_{A\in S_{\lambda,\Lambda}}\frac{\A(n,-\sigma)}{2}\int_{\R^n}\delta(u,x,y)\frac{y^TAy}{|y|^{n+2+\sigma}}dy
\end{equation*}
and
\begin{equation*}
\M^-u(x):=\inf_{A\in S_{\lambda,\Lambda}}\Tr(AD^\sigma u(x))=\inf_{A\in S_{\lambda,\Lambda}}\frac{\A(n,-\sigma)}{2}\int_{\R^n}\delta(u,x,y)\frac{y^TAy}{|y|^{n+2+\sigma}}dy.
\end{equation*}
Typical properties of $I$ are concavity and that
\[
\M^-(u-v)(x)\leq Iu(x)-Iv(x)\leq\M^+(u-v)(x)
\]
for bounded smooth functions $u$ and $v$. 
%
%
%
We recall the definition of viscosity solutions. We say that $\phi$ touches $u$ from below at $x$ whenever
\begin{equation}\label{eq23}
u(x)=\phi(x)\quad\mbox{and}\quad u(y)\geq \phi(y)\quad\mbox{for }y\in N,
\end{equation}
where $N$ is a neighborhood around $x$.
We have the analogous definition of touching from above.

\begin{dfn}\label{visco}
Let $\Omega\subset\R^n$ be open.
$u\in LSC(\Omega)\cap L^\infty(\R^n)$ (resp., $USC(\Omega) \cap L^\infty(\R^n)$) is a  viscosity supersolution (resp., subsolution) of $Iu=f$ in $\Omega$ if whenever $\phi$ touches $u$ from below (resp., above) at $x\in  \Omega$ for $\phi \in C^2(\ol{N})$ and $N$ in \eqref{eq23}, then   
\[
v:=
\left\{
\begin{aligned}
&\phi&\quad&\mbox{in}\ N,\\
&u&\quad &\mbox{in}\ \R^n\setminus N
\end{aligned}
\right.
\]
satisfies that $Iv(x)\leq f(x)$ (resp., $Iv(x)\geq f(x)$).

We say $u\in C(\Omega)\cap L^\infty(\R^n)$ is a viscosity solution of $Iu(x)=f(x)$ in $\Omega$ if $u$ is a viscosity supersolution and subsolution of $Iu(x)=f(x)$ in $\Omega$.
\end{dfn}

We next state several properties of the inf-convolution of $u$,
\begin{equation}\label{uh}
u_h(x):=\inf_{y\in\R^n}\left\{u(y)+\frac{1}{2h}|x-y|^2\right\},
\end{equation}
which are based on a classical results \cite{Jensen} for second order settings (see also Appendix in \cite{CIL} and Lemma A.2 in \cite{CCKS}).
To this end, we need to introduce some notations:
\begin{dfn}\label{def21}

\begin{enumerate}
\item We say operator $I$ is translation invariant if $\tau_z [Iu](x)=I[\tau_zu](x)$ holds for the translation operator $\tau_zu(x):=u(x+z)$ for $z\in\R^n$.
\item Let $\phi\in C^\infty_0(B_1)$ satisfying $\phi\geq0$ and $\|\phi\|_1=1$.
Then, for $\epsilon>0$ and $u\in LSC(\R^n)\cap L^\infty(\R^n)$, we define the standard modification of $u$ by
$
u_{h,\e}:=u_h*(\e^{-n}\phi(\e\cdot)).$
\end{enumerate}
\end{dfn}
\begin{prop}\label{infconv}
Let $I$ be a translation invariant nonlocal operator. 
Let $u\in LSC(\R^n)\cap L^\infty(\R^n)$ be a viscosity supersolution of $Iu=f$ in $\R^n$.
Then following properties hold:
\begin{enumerate}
\item[(i)] $Iu_h(x)$ is defined a.e. $x\in\R^n$.
\item[(ii)] $u_h$ is a viscosity supersolution of
\[
Iu_h(x)\leq \max_{|x-y|\leq 2\sqrt{h\|u\|_\infty}}f(y)\quad\mbox{in }\R^n.
\]
\item[(iii)] Let $u_{h,\e}$ be the standard modification of $u_h$. Then, there exists a constant $C>0$ such that $Iu_{h,\e}(x)\leq C\max\{h^{-1},\|u\|_\infty\}$ in $\R^n$ for $\e>0$. Moreover,
\[
Iu_{h,\e}(x)\to Iu_h(x)\quad\mbox{a.e. in }\R^n\mbox{ as }\e\to0. 
\]
\end{enumerate}
\end{prop}
\noindent
We give a proof of proposition \ref{infconv} in the Appendix.

Let us introduce the Riesz potential $P$ of $v\in C_0(\R^n)$,
\begin{equation}\label{P}
P(x):=\A(n,2-\sigma)\int_{\R^n}\frac{v(y)}{|x-y|^{n-(2-\sigma)}}dy,\quad(x\in\R^n),
\end{equation}
where $\A(n,2-\sigma)$ is from \eqref{A}.
We will use properties of $P$ proved in~\cite{GS12}.

\begin{prop}[Lemma 4.11 in~\cite{GS12}]\label{infP}
For $\sigma\in(0,2)$, there exists $M_0=M_0(n,\sigma)>1$ such that for any $r>0$ and non-positive function $v\in C_0(B_r)$, the Riesz potential $P$ of $v$ satisfies
\[
-\inf_{\R^n\setminus B_{M_0r}}P\leq-\frac{1}{2}\inf_{B_{M_0r}}P.
\]
\end{prop}

\begin{prop}[Lemma 5.1 in~\cite{GS12}]\label{Hessian}
Let $\Omega\subset\R^n$ be open. Assume that there exists a constant $C>0$ such that $v\in C_0(\R^n)$ satisfies
\[
\int_{\R^n}|\delta(v,x,y)|\frac{dy}{|y|^{n+\sigma}}\leq C\quad\mbox{in }\Omega.
\]
Then, $P\in C^{1,1}(\Omega)$ and the following formula holds
\begin{align*}
&\quad D^2P(x)\\
&=\frac{(n+\sigma-2)(n+\sigma)\A(n,2-\sigma)}{2}
\int_{\R^n}\delta(v,x,y)
\left[
\frac{y\otimes y}{|y|^{n+2+\sigma}}-\frac{Id}{(n+\sigma)|y|^{n+\sigma}}
\right]
dy
\end{align*}
for a.e. $x\in \Omega$. In particular, since $\A(n,-\sigma)=\sigma(n+\sigma-2)\A(n,2-\sigma)$, we have
\begin{equation*}
[D^2P(x)]_\sigma=D^\sigma v(x)\quad\mbox{and}\quad
\Tr(A_\sigma D^2P(x))=\Tr(AD^\sigma v(x))
\end{equation*}
for a.e. $x\in\Omega$, where $[D^2P(x)]_\sigma$ and $A_\sigma$ are from \eqref{As}.
\end{prop}


\begin{rem}\label{rem21}
For Lemma 4.11, Guillen-Schwab prove the case of $r=3$ in \cite{GS12}, but with scaling we can prove the case of general radius as well.
In addition, according to their proof, $M_0>1$ should satisfy
\[
\left(\frac{3M_0-3}{6}\right)^{-n+(2-\sigma)}\leq\frac{1}{2}.
\]
Therefore, in the case of $n\geq3$, $M_0$ can be chosen independently of $\sigma\in(0,2)$.
\end{rem}

Next, we recall $W^{2,p}$~estimates for second order equations in Theorem 1 of~\cite{Caf} and Theorem 1 of~\cite{Esc}.
We point out that settings of~\cite{Caf,Esc} is much more general than those of the next theorem.
However we restrict ourselves, in order to simplify our statements in the proceeding argument.
\begin{thm}[$W^{2,p}$~estimates]\label{W2p}
There exists $p^*=p^*(n,\lambda,\Lambda)\in[n/2,n)$ such that for any $p\in(p^*,\infty)$, there exist $C_1>0$ and $\theta_1\in(0,1)$ depending only on $n$, $\lambda$, $\Lambda$ and $p$ with the following property: Let $f\in L^p(B_3)$, $\sigma\in(0,2)$ and $\mathscr{B}$ be an index set. Suppose $A_\beta(x)\in S_{\lambda,\Lambda}$ satisfy for any $B_r(x)\subset B_3$,
\begin{equation}\label{theta}
\frac{1}{r^n}\int_{B_r(x)}\sup_{\beta\in\mathscr{B}}|A_\beta(x)-A_\beta(y)|^ndy\leq\theta_1.
\end{equation}
If $u\in W^{2,p}_{loc}(B_3)\cap L^p(B_3)$ is a $L^p$-strong solution of
\[
\inf_{\beta\in\mathscr{B}}\Tr([A_\beta(x)]_\sigma D^2u(x))
=f(x)\quad\mbox{in }B_3,
\]
where $[A_\beta(x)]_\sigma$ is from \eqref{As}, then
\[
\|u\|_{W^{2,p}(B_2)}\leq C_1(\|f\|_{L^p(B_3)}+\|u\|_{L^p(B_3)}).
\]
\end{thm}
By a covering argument, we derive the following corollary:
\begin{coro}\label{coro1}
Let $p\in(p^*,\infty)$. Let $\Omega\subset\R^n$ be bounded and open and $\Omega'\Subset\Omega$.
Then, there exist $C_2=C_2(n,\lambda,\Lambda,p,\mathrm{dist}(\Omega',\partial\Omega))>0$ with the following property: Let $f\in L^p(\Omega)$, $\sigma\in(0,2)$ and $\mathscr{B}$ be an index set.
Suppose $A_\beta(x)\in S_{\lambda,\Lambda}$ satisfy \eqref{theta} for any $B_r(x)\subset \Omega$.
If $u\in W^{2,p}_{loc}(\Omega)\cap L^p(\Omega)$ is a $L^p$-strong solution of
\[
\inf_{\beta\in\mathscr{B}}\Tr([A_\beta(x)]_\sigma D^2u(x))
=f(x)\quad\mbox{in }\Omega,
\]
where $[A_\beta(x)]_\sigma$ is from \eqref{As}, then
\[
\|u\|_{W^{2,p}(\Omega')}\leq C_2(\|f\|_{L^p(\Omega)}+\|u\|_{L^p(\Omega)}).
\]
\end{coro}

We also recall Morrey's inequality (see Theorem 7.26 in~\cite{GT}).
\begin{thm}\label{Morrey}
Assume $p\in (n/2,\infty)$ and $p\neq n$. Then there exists a constant $C_2=C_2(n,p)>0$ such that
\begin{equation*}
\|u\|_{C^{2-n/p}(\overline{B_2})}\leq C_3\|u\|_{W^{2,p}(B_2)}.
\end{equation*}
\end{thm}


\section{ABP maximum principles}\label{S3}
In this section, we derive the following ABP maximum principle:
\begin{thm}\label{ABP}
Let $p^*$ be from Theorem \ref{W2p}. Assume $p\in (p^*,\infty)$ and $p\neq n$.
There exists a constant $C_4=C_4(n,\lambda,\Lambda,p)>0$ such that if $u\in LSC(\overline{B_1})\cap L^\infty(\R^n)$ is a viscosity supersolution of $M^-u\leq f$ in $B_1$ with $f\in C(\overline{B_1})$, $\sigma\in (n/p,2)$ and $u\geq0$ in $\R^n\setminus B_1$, then it follows that
\begin{equation*}
-\inf_{B_1}u\leq \frac{C_4\sigma M_0^{2-n/p}}{\sigma-n/p}\|f^+\|_{L^p(B_1)},
\end{equation*}
where $M_0$ is from Proposition \ref{infP}.
\end{thm}

\begin{rem}
In the case that $\sigma\in(1,2)$ and $p=n$, we also have the ABP maximum principle in the form
\[
-\inf_{B_1}u\leq \tilde{C}_4\|f^+\|_{L^n(B_1)},
\]
where $\tilde{C}_4$ depends only on $n,\lambda$, $\Lambda$ and $\sigma$.
\end{rem}

\begin{rem}
Although we assume $f\in C(\overline{B_1})$ in Theorem \ref{ABP}, it still holds even if we assume $f\in C(B_1)\cap L^p(B_1)$.
Indeed, this is observed as follows:
For any $\eta>0$, there exists $x_0\in B_1$ such that $u(x_0)-\eta\leq\inf_{B_1}u$.
We may choose $r\in(0,1)$ such that $x_0\in B_r$.
Then, we note $f\in C(B_1)\subset C(\overline{B_r})$ and apply Theorem \ref{ABP} to $u_r(x):=u(rx)-\inf_{\R^n\setminus B_r}u$, which is a viscosity supersolution of
\[
\M^-u_r(x)\leq r^\sigma f(rx)\quad\mbox{in }B_1
\]
and $u_r\geq0$ in $\R^n\setminus B_1$.
Then, we derive
\begin{align*}
-u(x_0)+\inf_{\R^n\setminus B_r}u
\leq-\inf_{B_1}u_r\leq \frac{C_4\sigma M_0^{2-n/p}}{\sigma-n/p}r^{\sigma-n/p}\|f^+\|_{L^p(B_r)},
\end{align*}
where we applied $x_0\in B_r$ and
$
-u(x_0)
\leq -\inf_{B_r}u=-\inf_{B_1}u_r-\inf_{\R^n\setminus B_r}u
$ to the last inequality.
We complete the proof by letting $r\to1$.
\end{rem}
\begin{rem}
As we pointed out in Remark \ref{rem21}, we can choose $M_0$ independently of $\sigma\in(0,2)$ if $n\geq3$.
Hence we obtain an improved constant $\tilde{C}=\tilde{C}(n,\lambda,\Lambda,p)>0$ in Theorem \ref{ABP} such that
\begin{equation*}
-\inf_{B_1}u\leq \frac{\tilde{C}\sigma}{\sigma-n/p}\|f^+\|_{L^p(B_1)}
\end{equation*}
for $u$ as in Theorem \ref{ABP}.
On the other hand, in the case of $n=2$, a more detailed discussion is needed to improve the ABP maximum principle, which we will conduct in the future work. 
\end{rem}

Let us first, construct an approximating equation and its solution, that will be used in the proof of Theorem \ref{ABP}.
An approximation of $\M^-$ is defined by
\[
\M^-_\eta u(x):=\inf_{A\in S_{\lambda,\Lambda}\mbox{ and }A\geq\eta Id}\Tr(AD^\sigma u(x)).
\]
Let $g\in C^\infty_0(B_3)$ be an arbitrary function satisfying
\begin{equation}\label{fg}
0\leq f^+(x)\chi_{B_1}(x)\leq g(x)\quad\mbox{in }\R^n,
\end{equation}
where $\chi_{B_1}$ is the indicator function of $B_1$.
We shall consider a unique classical solution $v\in C^{\sigma+\alpha}(B_3)\cap C(\R^n)$ of
\begin{equation}\label{veq}
\left\{
\begin{aligned}
\M^-_\eta v(x)&=g(x)&\quad&\mbox{in }B_3,\\
v(x)&=0&\quad&\mbox{in }\R^n\setminus B_3,
\end{aligned}
\right.
\end{equation}
where $\alpha\in(0,1)$ depends only on $n$, $\lambda$, $\Lambda$ and $\eta$.
We intend to estimate the $L^\infty$-norm of $v$; roughly, this will lead to a lower bound of a supersolution $u$, according to comparison principles.
\begin{rem}\label{rem35}
The existence and uniqueness of $v$ follow form Theorem 1.3 in~\cite{Serra}.
In addition, $v$ is non-positive in $\R^n$ due to the comparison between $v$ and $\tilde{v}\equiv0$, which is a solution of $\M_\eta^-\tilde{v}=0$ in $\R^n$ (see Theorem 5.2 \cite{CS09} for the comparison principle).
\end{rem}


Now, we show the following two key estimates.
\begin{lem}\label{lem31}
Assume $p\in(p^*,\infty)$ and $p\neq n$.
There is a constant $C_5=C_5(n,\lambda,\Lambda,p)>0$ such that if $v\in C^{\sigma+\alpha}(B_3)\cap C(\R^n)$ is a classical solution of \eqref{veq}, then it follows that
\begin{equation*}
\|v\|_{L^\infty(B_1)}\leq \frac{C_5\sigma}{\sigma-n/p}(\|g\|_{L^p(B_3)}+\|P\|_{L^\infty(\R^n)}),
\end{equation*}
where $P$ is the Riesz potential of $v$, defined by \eqref{P}.
\end{lem}

\begin{proof}
We first observe that according to Proposition \ref{Hessian}, $P$ solves the second order equation
\begin{align}\nonumber
\inf_{A\in S_{\lambda,\Lambda}\mbox{ and }A\geq\eta Id}\Tr(A_\sigma D^2 P(x))=\M^-_\eta v(x)
=g(x)\quad\mbox{in }B_3,
\end{align}
where $A_\sigma$ is from \eqref{As}.
Hence, Theorem \ref{W2p} and \ref{Morrey} imply that $P$ satisfies
\begin{equation}
\|P\|_{C^{2-n/p}(\overline{B_2})}\leq C_2\|P\|_{W^{2,p}(B_2)}\leq C_1C_2(\|g\|_{L^p(B_3)}+\|P\|_{L^p(B_3)}).\label{eq1lem31}
\end{equation}
On the other hand, since the fractional Laplacian: $(-\Delta)^{(2-\sigma)/2}$ is the left inverse to the Riesz potential operator (see Theorem 3.22 in \cite{Sam} for instance), we have the following identity:
\begin{equation}\label{eq3lem31}
v(x)=\frac{\A(n,-(2-\sigma))}{2}\int_{\R^n}-\delta(P,x,y)\frac{dy}{|y|^{n+(2-\sigma)}}\quad\mbox{in }B_1,
\end{equation}
where $\A(n,-(2-\sigma))$ is defined by \eqref{-A} with $\sigma$ replaced by $2-\sigma$.
Noting $x+y,x-y\in B_2$ for $x,y\in B_1$, we estimate the right hand side of \eqref{eq3lem31} as follows:
\begin{align*}
&\quad\left|\frac{\A(n,-(2-\sigma))}{2}\int_{B_{1}}-\delta(P,x,y)\frac{dy}{|y|^{n+(2-\sigma)}}\right|\\
&\leq \frac{\A(n,-(2-\sigma))}{2}\|P\|_{C^{2-n/p}(B_{2})}\int_{B_{1}}|y|^{-n-n/p+\sigma}dy\\
&=\frac{\A(n,-(2-\sigma))}{2}\|P\|_{C^{2-n/p}(B_{2})}\int_0^{1}r^{-n-n/p+\sigma}|\partial B_r|dr\\
&=\frac{\A(n,-(2-\sigma))|\partial B_1|}{2(\sigma-n/p)}\|P\|_{C^{2-n/p}(B_{2})}\\
&\leq\frac{C\sigma(2-\sigma)}{(\sigma-n/p)}\|P\|_{C^{2-n/p}(B_{2})}
\end{align*}and
\begin{align*}
&\quad\left|\frac{\A(n,-(2-\sigma))}{2}\int_{\R^n\setminus B_{1}}-\delta(P,x,y)\frac{dy}{|y|^{n+(2-\sigma)}}\right|\\
&\leq \frac{\A(n,-(2-\sigma))}{2}\|P\|_{L^\infty(\R^n)}\int_{\R^n\setminus B_{1}}4|y|^{-n-(2-\sigma)}dy\\
&=\frac{\A(n,-(2-\sigma))}{2}\|P\|_{L^\infty(\R^n)}\int_{1}^\infty 4r^{-n-(2-\sigma)}|\partial B_r|dr\\
&=\frac{2\A(n,-(2-\sigma))|\partial B_1|}{(2-\sigma)}\|P\|_{L^\infty(\R^n)}\\
&\leq C\sigma\|P\|_{L^\infty(\R^n)}
\end{align*}
where $C$ depends only on $n$ since, taking account of
\begin{align*}
\lim_{\sigma\to0+0}\frac{\A(n,-(2-\sigma)}{\sigma(2-\sigma)}=\frac{\Gamma((n+2)/2)}{\pi^{n/2}},\quad\lim_{\sigma\to2-0}\frac{\A(n,-(2-\sigma)}{\sigma(2-\sigma)}=\frac{\Gamma(n/2)}{4\pi^{n/2}},
\end{align*}
$\A(n,-(2-\sigma))/(\sigma(2-\sigma))$ is bounded from below and above for $\sigma\in(0,2)$.

Thus, combine \eqref{eq1lem31} and the above inequalities, to compute
\begin{align*}
\|v\|_{L^\infty(B_1)}
&\leq C\sigma\left(\frac{(2-\sigma)}{\sigma-n/p}\|P\|_{C^{2-n/p}(B_{2})}+\|P\|_{L^\infty(\R^n)}\right)\\
&\leq \frac{C\sigma(2-n/p)}{\sigma-n/p}\left(\|P\|_{C^{2-n/p}(B_{2})}+\|P\|_{L^\infty(\R^n)}\right)\\
&\leq \frac{\tilde{C}\sigma(2-n/p)}{\sigma-n/p}\left(\|g\|_{L^{p}(B_{2})}+\|P\|_{L^\infty(\R^n)}\right),
\end{align*}

which proves this lemma for $C_5:=\tilde{C}(2-n/p)$.
\end{proof}

\begin{lem}\label{lem32}
There exists a constant $C_6=C_6(n,\lambda,\Lambda,p)>0$ such that if $v\in C^{\sigma+\alpha}(B_3)\cap C(\R^n)$ is a classical solution of \eqref{veq} and $P$ is the Riesz potential of $v$, then it follows that
\[
\|P\|_{L^\infty(\R^n)}\leq C_6M_0^{2-n/p}\|g\|_{L^p(B_3)}.
\]
\end{lem}

\begin{proof}
Note that $v$ is non-positive as we mentioned in Remark \ref{rem35}.
Hence from the fact that $P$ is also non-positive in $\R^n$ and Proposition \ref{infP}, it suffices to see
\begin{equation}\label{eq1lem32}
-\inf_{\R^n}P=-\inf_{B_{3M_0}}P\leq C_6M_0^{2-n/p}\|g\|_{L^p(B_3)}.
\end{equation}
If $p=n$, this is already proved by~\cite{GS12}.
We will confirm this in the general case.

Since $v$ is non-positive in $\R^n$ and supported by $B_3$, we have $\delta(v,x,y)\leq0$ for $x\in \R^n\setminus B_3$ and $y\in \R^n$, which implies that $v$ is a viscosity supersolution of
\begin{equation*}
\M^-_\eta v(x)\leq g(x)\quad\mbox{in }\R^n.
\end{equation*}
Let $v_h$ be the inf-convolution of $v$, which is a viscosity supersolution of
\[
\M^-_\eta v_h(x)\leq \max_{|x-y|\leq 2\sqrt{h\|u\|_\infty}}\{ g(y)\}=:g_h(x)\quad\mbox{in }\R^n.
\]
We also denote by $v_{h,\e}$ the standard modification of $v_h$, and by $P_h$ and $P_{h,\e}$ Riesz potentials of $v_h$ and $v_{h,\e}$ respectively.
We may select $h,\e>0$ so small that $v_h$ and $v_{h,\e}$ are supported by $B_4$.
Then, from Proposition \ref{Hessian}, $P_{h,\e}$ solves
\[
\inf_{A\in S_{\lambda,\Lambda}\mbox{ and }A\geq\eta Id} \Tr(A_\sigma D^2P_{h,\e}(x))=\M^-_\eta v_{h,\e}(x)\quad\mbox{in }\R^n.
\]
Hence, we employ the ABP maximum principle of second order equations (see (2.3) in~\cite{CCKS}), to derive
\[
-\inf_{B_{3M_0}}P_{h,\e}\leq-\inf_{\partial B_{3M_0}}P_{h,\e}+C(3M_0)^{2-n/p}\|(\M^-_\eta v_{h,\e})^+\|_{L^p(B_{R_0})}.
\]
Therefore, Proposition \ref{infP} gives
\[
-\inf_{B_{R_0}}P_{h,\e}\leq \tilde{C}M_0^{2-n/p}\|(\M^-_\eta v_{h,\e})^+\|_{L^p(B_{R_0})}.
\]
Noting $(\M^-_\eta v_{h,\e})^+$ is uniformly bounded for $\e>0$ from (iii) in Proposition \ref{infconv}, we send $\e\to0$ to find
\begin{equation*}
-\inf_{B_{R_0}}P_{h}\leq \tilde{C}M_0^{2-n/p}\|(\M^-_\eta v_{h})^+\|_{L^p(B_{R_0})}\leq \tilde{C}M_0^{2-n/p}\|g_h\|_{L^p(B_{R_0})}.
\end{equation*}
Consequently, let $h\to0$, which proves \eqref{eq1lem32}. 
\end{proof}

We combine Lemma \ref{lem31} and \ref{lem32}, to see
\begin{align}
\|v\|_{L^\infty(B_1)}
&\leq\frac{C_5\sigma}{\sigma-n/p}(\|g\|_{L^p(B_3)}+\|P\|_{L^\infty(\R^n)})\nonumber\\
&\leq\frac{C\sigma M_0^{2-n/p}}{\sigma-n/p}\|g\|_{L^p(B_3)},\label{infv}
\end{align}
where applied $M_0>1$ to the second inequality.
Employing \eqref{infv}, we complete the proof of Theorem \ref{ABP}.

\noindent{\bf Proof of Theorem \ref{ABP}.}
We start to modify a supersolution $u$ similarly to Lemma \ref{lem32}. 
It is standard that $-u^-$ is a viscosity supersolution of
\[
\M^-[-u^-](x)\leq f^+(x)\chi_{B_1}(x)\quad\mbox{in }\R^n.
\]
Let $(-u^-)_h$ be the inf-convolution of $-u^-$ and $(-u^-)_{h,\e}$ be the standard modification of $(-u^-)_h$.
Then, $(-u^-)_h$ is a viscosity supersolution of
\begin{equation}\nonumber
\M^-[(-u^-)_h](x)\leq\max_{|x-y|\leq 2\sqrt{h\|u\|_\infty}}\{ f^+(y)\chi_{B_1}(y)\}=:f^+_h(x)\quad\mbox{in }\R^n.
\end{equation}
Since $-u^-$ is supported in $B_1$, we may fix $h,\e>0$ so small that $(-u^-)_h$ and $(-u^-)_{h,\e}$ are supported in $B_3$.

We next consider an arbitrary function $g\in C^\infty_0(B_3)$ satisfying
\begin{equation}\label{eq1thm31}
0<(\M^-_\eta [(-u^-)_{h,\e}](x))^+\leq g(x)\quad\mbox{in }\R^n,
\end{equation}
and a unique classical solution $v$ of \eqref{veq}.
Then, we observe that
\[
\M^-_\eta [(-u^-)_{h,\e}]\leq g=\M^-_\eta v\quad\mbox{in }B_3\quad\mbox{and}\quad(-u^-)_{h,\e}=v=0\quad \mbox{in }\R^n\setminus B_3,
\]
and so $v\leq (-u^-)_{h,\e}$ in $B_3$ according to comparison principles (see Theorem 5.2 in~\cite{CS09}).
Consequently, it follows from \eqref{infv},
\[
-\inf_{B_1}(-u^-)_{h,\e}\leq\|v\|_{L^\infty(B_1)} \leq \frac{C\sigma M_0^{2-n/p}}{\sigma-n/p}\|g\|_{L^p(B_3)}.
\]
Choose $g=g_k$ such that $g_k(x)\searrow(\M^-_\eta [(-u^-)_{h,\e}](x))^+$ for $x\in\R^n$ as $k\to\infty$.
Then, we derive
\begin{align*}
&\quad-\inf_{B_1}(-u^-)_{h,\e}\leq \frac{C\sigma M_0^{2-n/p}}{\sigma-n/p}\|(\M^-_\eta[(- u^-)_{h,\e}])^+\|_{L^p(B_3)}.
\end{align*}
Note that $(\M^-_\eta[(- u^-)_{h,\e}])^+$ is bounded from above for $\eta,\e>0$ by (iii) of Proposition\ref{infconv}.
Hence, by letting $\eta\to0$ and $\e\to0$, the dominate convergence theorem shows
\[
-\inf_{B_1}(-u^-)_{h}\leq \frac{C\sigma M_0^{2-n/p}}{\sigma-n/p}\|(\M^- [(-u^-)_{h}])^+\|_{L^p(B_3)}\leq \frac{C\sigma M_0^{2-n/p}}{\sigma-n/p}\|f_h^+\|_{L^p(B_3)}.
\]
Finally, passing to limits as $h\to0$, we prove Theorem \ref{ABP}.$\quad\Box$


\section{$W^{\sigma,p}$~estimates}\label{S4}
This section present the following $W^{\sigma,p}$~estimate:
\begin{thm}\label{thm41}
Let $p^*$ be from Theorem \ref{W2p} and $p\in(p^*,\infty)$.
There exists a constant $C_7>0$ depending only on $n,\lambda,\Lambda$ and $p$ with the following property:
Let $f\in C(\overline{B_3})$ and $\mathscr{B}$ be an index set.
For $\beta\in\mathscr{B}$, $A_\beta(x)\in S_{\lambda,\Lambda}$ have a uniform modulus of continuity and satisfy \eqref{theta}.
If $u\in C(\R^n)\cap L^\infty(\R^n)$ is a viscosity solution of \eqref{eq1} with $\sigma\in(n/p,2)$, then it follows that
\begin{equation}\label{eq1thm41}
\|D^\sigma u\|_{L^p(B_2)}\leq C_7(\|f\|_{L^p(B_3)}+\|u\|_{L^\infty(\R^n)}).
\end{equation}
\end{thm}
\begin{rem}
Since $A_\beta$ has the uniform modulus of continuity, one may prove an estimate without \eqref{theta}. However, due to \eqref{theta}, 
$C_6$ depends on $\theta_1$ of \eqref{theta}, but the modulus of continuity.
\end{rem}
\begin{rem}
It is valid that \eqref{eq1thm41} holds for any $p\in(1,\infty)$ if we take the fractional Laplacian $-(-\Delta)^{\sigma/2}$ as the nonlocal operator $I$ in \eqref{eq1}.
In fact, according to $L^p$-estimates of Poisson's equation, we observe
\[
\|D^\sigma u\|_p=\|[D^2P]_\sigma\|_p\leq C\|\Delta P\|_p=C\|-(-\Delta)^{\sigma/2} u\|_p.
\]
\end{rem}
\begin{proof}
We first approximate our equation as follows.
Let $A_{\beta,\rho}$ be a smooth function such that $A_{\beta,\rho}(x)\in S_{\lambda,\Lambda}$, $A_{\beta,\rho}(x)\geq\rho Id$ and 
\[
\sup_{\beta\in\mathscr{B}}\|A_{\beta,\rho}- A_\beta\|_{L^\infty(B_3)}\to0\quad\mbox{as }\rho\to0.
\]
Also, set an approximation $f_\rho$ of $f$ by the standard modifier.
Then for some $\alpha\in(0,1)$, there exists a unique classical solution $u_\rho\in C^{\sigma+\alpha}(B_3)\cap C(\R^n)\cap L^\infty(\R^n)$ of
\begin{equation}
\left\{
\begin{aligned}\nonumber
I_\rho u_\rho&=f_\rho&\quad&\mbox{in }B_3,\\
u_\rho&=u&\quad&\mbox{in }\R^n\setminus B_3,
\end{aligned}
\right.
\end{equation}
where $I_\rho$ is of form \eqref{operator} replacing $A_\beta$ by $A_{\beta,\rho}$.
Here, the existence and uniqueness of $u_\rho$ follow from Theorem 1.3 in~\cite{Serra}.

Now, we need to show \eqref{eq1thm41} for $u_\rho$ instead of $u$.  
Setting $w_\rho:=u_\rho\chi_{B_4}$, we compute
\begin{align*}
|I_\rho w_\rho (x)|&\leq |I_\rho u_\rho(x)|+\A(n,-\sigma)\Lambda\int_{\R^n\setminus B_1}|u_\rho(x+y)|\frac{dy}{|y|^{n+\sigma}}\\
&\leq |f_\rho(x)|+C\|u_\rho\|_{L^\infty(\R^n)}\quad\mbox{in }B_3,
\end{align*}
and
\begin{align}
\|D^\sigma u_\rho\|_{L^p(B_2)}&\leq\|D^\sigma w_\rho\|_{L^p(B_2)}+\|D^\sigma (u_\rho-w_\rho)\|_{L^p(B_2)}\nonumber\\
&\leq\|D^\sigma w_\rho\|_{L^p(B_2)}+C\|u_\rho\|_{L^\infty(\R^n)}.\label{eq2thm41}
\end{align}
Let $P_\rho$ be the Riesz potential of $w_\rho$ defined by \eqref{P}, where $v$ is replaced by $w_\rho$.
Then, from Proposition \ref{Hessian}, $P$ satisfies
\[
\left|\inf_{a\in\mathcal{A}}\Tr([A_{\beta,\rho}(x)]_\sigma D^2P_{\rho}(x))\right|=|I_\rho w_\rho(x)|\leq |f_\rho(x)|+C\|u_\rho\|_{L^\infty(\R^n)},
\]
where $[A_{\beta,\rho}(x)]_\sigma$ is from \eqref{As}.
Thus, Theorem \ref{W2p} implies
\begin{equation*}
\|D^\sigma w_\rho\|_{L^p(B_2)}=\|[D^2P_\rho(\cdot)]_\sigma\|_{L^p(B_2)}\leq C(\|f_\rho\|_{L^p(B_3)}+\|u_\rho\|_{L^\infty(\R^n)}).
\end{equation*}
Consequently, the above and \eqref{eq2thm41} imply that \eqref{eq1thm41} holds for $u_\rho$.
Similarly, by applying Corollary \ref{coro1}, we also have for $0<r<3$,
\begin{align}
\|P_\rho\|_{W^{2,p}(B_r)}&\leq C_r(\|f_\rho\|_{L^p(B_3)}+\|P_\rho\|_{L^p(B_3)})\quad\mbox{and}\nonumber\\
\|D^\sigma u_\rho\|_{L^p(B_r)}&\leq C_r(\|f_\rho\|_{L^p(B_3)}+\|u_\rho\|_{L^\infty(\R^n)}),\label{eq4thm41}
\end{align}
where we take $C_r:=C_2(n,\lambda,\Lambda,p,3-r)$ as in Corollary \ref{coro1}.

Next, we claim that $u_\rho\to u$ uniformly in $B_3$ as $\rho\to0$.
To see this, we first observe that $u-u_\rho$ is a viscosity supersolution of
\begin{align*}
  \M^-[u-u_\rho](x)&\leq Iu(x)-Iu_\rho(x)\\
  &\leq |f(x)-f_\rho(x)|+|I u_\rho(x)-I_\rho u_\rho(x)|\\
  &\leq|f(x)-f_\rho(x)|+\sup_{\a\in\mathcal{A}}|A_\beta(x)-A_{\beta,\rho}(x)||D^\sigma u_\rho(x))|\quad\mbox{in }B_3.
\end{align*}
Therefore, by applying Theorem \ref{ABP} to $u-u_\rho-\inf_{B_3\setminus B_r}(u-u_\rho)$ in $B_r\subset B_3$, we derive
\begin{align*}
  &\quad-\inf_{B_r}(u-u_\rho)+\inf_{B_3\setminus B_r}(u-u_\rho)\\
  &\leq C_2r^{\sigma-n/p}\left\{\|f-f_\rho\|_{L^p(B_r)}+\sup_{a\in\mathcal{A}}\|A_\beta-A_{\beta,\rho}\|_{L^\infty(B_r)}\|D^\sigma u_\rho\|_{L^p(B_r)}\right\}.
\end{align*}
Take the limsup for $\rho\to0$ to see 
\[
  \limsup_{\rho\to0}\left(-\inf_{B_r}(u-u_\rho)\right)\leq\limsup_{\rho\to0}\left(-\inf_{B_3\setminus B_r}(u-u_\rho)\right)
\]
and hence
\begin{align*}
   \limsup_{\rho\to0}\left(-\inf_{B_3}(u-u_\rho)\right)&\leq\limsup_{\rho\to0}\left(-\inf_{B_r}(u-u_\rho)-\inf_{B_3\setminus B_r}(u-u_\rho)\right)\\
   &\leq2\limsup_{\rho\to0}\left(-\inf_{B_3\setminus B_r}(u-u_\rho)\right)
\end{align*}
for any $r>0$.
Note that a uniform modulus of continuity of $u_\rho$ and $u$ on $\partial B_3$ can be obtained by the same way as that of Lemma~2 in~\cite{CS11b} once we observe that Lemma 1 of \cite{CS11b} still holds in our setting, or more precisely $\phi(x):=((|x|-3)^+)^{\alpha}$ satisfies $\M^+\phi\leq0$ in $B_{3+r}\setminus B_{3}$ for small $\alpha,r>0$, which we omit here.
Hence there exists a modulus of continuity $\omega$ such that
\begin{align*}
|u_\rho(x)-u_\rho(y)|\leq \omega(|x-y|)\quad\mbox{and}\quad|u(x)-u(y)|\leq \omega(|x-y|)
\end{align*}
for $\rho>0$, $x\in \R^n$ and $y\in\partial B_3$. Since $u_\rho(y)=u(y)$ on $\partial B_3$, we have for $x\in B_3\setminus B_r$,
\begin{align*}
|u_\rho(x)-u(x)|\leq |u_\rho(x)-u_\rho(y)|+|u(x)-u(y)|&\leq 2\omega(\mathrm{dist}(x,\partial B_3))\\
&\leq 2\omega(3-r),
\end{align*}
where we take $y\in\partial B_3$ such that $|x-y|=\mathrm{dist}(x,\partial B_3)$.
Hence
\[
\limsup_{\rho\to0}\left(-\inf_{B_3}(u-u_\rho)\right)\leq2\limsup_{\rho\to0}\left(-\inf_{B_3\setminus B_r}(u-u_\rho)\right)\leq4\omega(3-r).
\]
By letting $r\to3$, we have $\limsup_{\rho\to0}(-\inf_{B_3}(u-u_\rho))\leq0.$
On the other hand, we also have $\limsup_{\rho\to0}\sup_{B_3}(u-u_\rho)\leq0,$
repeating a similar argument with $-(u-u_\rho)$.
Hence $u_\rho$ uniformly converge to $u$ in $B_3$ as $\rho\to0$.
Since $u_\rho$ satisfies \eqref{eq1thm41} for $\rho>0$, $D^\sigma u_\rho\rightharpoonup D^\sigma u$ weakly in $L^p(B_2)$ and \eqref{eq1thm41} still holds for $u$.
\end{proof}


\section{Examples}\label{S5}

Theorem \ref{ABP} says that when we fix $\sigma\in (0,2)$, the ABP maximum principle in the form
\begin{equation}\label{eq1E}
-\inf_{B_1}u\leq-\inf_{\R^n\setminus B_1}u+C(n,\sigma,p)\|(\M^-u)^+\|_{L^p(B_1)}
\end{equation}
holds for $p>\max\{n/\sigma,p^*\}$.
In this section, we provide an example that only $p>n/\sigma$ is not enough to be that \eqref{eq1E} holds, by following~\cite{AIM,Pucci}.

Now, we fix $p_0$ such that
\[
p_0=\frac{n(\sigma+1)}{n+\sigma}.
\]
and assume $p_0>n/\sigma$, whence $n=2,3$ and $\sigma\in(\sqrt{n},2)$.
In order to see our assertion, we must construct a sequence $\{u_N\}_{N=1}^\infty$ such that
\[
\frac{-\inf_{B_1}u_N+\inf_{\R^n\setminus B_1}u_N}{\|(\M^-u_N)^+\|_{L^{p_0}(B_1)}}\to\infty\quad\mbox{as }N\to\infty.
\]
We first choose the Riesz potential $P_N$ of $u_N$. Let $\phi_N(r)\in C^{1,1}([0,\infty))$ be monotonically non-decreasing and $P_N(x):=\phi_N(|x|)$.
Then
\[
D^2P_N(x)=\phi_N''(|x|)\frac{x}{|x|}\otimes\frac{x}{|x|}+\frac{\phi_N'(|x|)}{|x|}\left\{Id-\frac{x}{|x|}\otimes\frac{x}{|x|}\right\}
\]
and
\begin{align}
&\quad\inf_{A\in S_{\lambda,\Lambda}}\Tr(A_\sigma D^2P_N(x))\nonumber\\
&=\inf_{A\in S_{\lambda,\Lambda}}\left\{\frac{\sigma}{\sigma+n}\Tr(A D^2P_N(x))+\frac{1}{\sigma+n}\Tr(A)\Tr(D^2P_N)\right\}\nonumber\\
&=\inf_{A\in S_{\lambda,\Lambda}}\biggl[\left\{\phi_N''(|x|)-\frac{\phi_N'(|x|)}{|x|}\right\}\frac{\sigma}{\sigma+n}\left\langle A \frac{x}{|x|}, \frac{x}{|x|}\right\rangle\nonumber\\
&\quad+\left\{\frac{1}{\sigma+n}\phi_N''(|x|)+\frac{\sigma+n-1}{\sigma+n}\frac{\phi_N'(|x|)}{|x|}\right\}\Tr(A)\biggr]\nonumber\\
&=\inf_{A\in S_{\lambda,\Lambda}}\biggl[\left\{\frac{\sigma+1}{\sigma+n}\phi_N''(|x|)+\frac{n-1}{\sigma+n}\frac{\phi_N'(|x|)}{|x|}\right\}\left\langle A \frac{x}{|x|}, \frac{x}{|x|}\right\rangle\nonumber\\
&\quad+\left\{\frac{1}{\sigma+n}\phi_N''(|x|)+\frac{\sigma+n-1}{\sigma+n}\frac{\phi_N'(|x|)}{|x|}\right\}\Tr \left(A\left(Id-\frac{x}{|x|}\otimes\frac{x}{|x|}\right)\right)\biggr].\label{eq0E}
\end{align}
Set
\[
\tau:=\frac{\sigma+1}{\sigma+n}.
\]
By choosing
\[
A=\lambda(n+\sigma)\frac{x}{|x|}\otimes \frac{x}{|x|}\in S_{\lambda,\Lambda},
\]
we see
\begin{equation}\label{eq2E}
\inf_{A\in S_{\lambda,\Lambda}}\Tr(A_\sigma D^2P_N(x))\leq \lambda(n+\sigma)\left\{\tau\phi_N''(|x|)+(1-\tau)\frac{\phi_N'(|x|)}{|x|}\right\}.
\end{equation}
Noting that $2\tau-1>0$, we define
\begin{equation*}
\phi_N(r)=
\left\{
\begin{aligned}
&0\quad\mbox{if }r\geq1,\\
&-\frac{(1-r)^2\log(1-\tau) N}{2\tau^2(1-\tau)^{(1-\tau)/\tau}}
\quad\mbox{if }1-\tau\leq r< 1,\\
&-\frac{1}{2\tau-1}\biggl\{\left(\tau(1-\tau)^{-(1-\tau)/\tau}\log(1-\tau) N-r^{(2\tau-1)/\tau}\log rN\right)\\
&\left.-\frac{\tau}{2\tau-1}\left((1-\tau)^{(2\tau-1)/\tau}-r^{(2\tau-1)/\tau}\right)\right\}\quad\mbox{if }1/N\leq r<(1-\tau),\\
&\phi_N(1/N)\quad\mbox{if }0\leq r<1/N.
\end{aligned}
\right.
\end{equation*}
Observe that $\phi_N\in C^{1,1}_0([0,\infty))$,
\[
\phi_N'(r)=
\left\{
\begin{alignedat}{3}
&\frac{(1-r)\log(1-\tau) N}{\tau^2(1-\tau)^{(1-\tau)/\tau}}& &\mbox{ for }1-\tau\leq r< 1,\\
&\frac{\log rN}{\tau r^{(1-\tau)/\tau}}&&\mbox{ for }1/N\leq r< 1-\tau,\\
&0& &\mbox{ for }0\leq r< 1/N,r\geq1
\end{alignedat}
\right.
\]
and
\[
\phi_N''(r)=
\left\{
\begin{alignedat}{3}
&-\frac{\log(1-\tau) N}{\tau^2(1-\tau)^{(1-\tau)/\tau}}& &\mbox{ for }1-\tau\leq r< 1,\\
&-\frac{1}{\tau r^{1/\tau}}\left(\frac{1-\tau}{\tau}\log rN-1\right)&&\mbox{ for }1/N\leq r< 1-\tau,\\
&0& &\mbox{ for }0\leq r<1/N,r\geq1.
\end{alignedat}
\right.
\]
Set
\[
u_N(x):=\A(n,-(2-\sigma))\int_{\R^n}\left(P_N(x)-P_N(x+y)\right)\frac{dy}{|y|^{n+(2-\sigma)}},
\]
where $A(n,-(2-\sigma))$ is as in \eqref{-A} with $\sigma$ replaced by $2-\sigma$.
Notice that $u_N\in C^{\sigma}(\R^n)$ according to $P_N\in C^{1,1}_0(\R^n)$ and local H\"older estimates of Proposition 2.5 in~\cite{Sil}.
Since $P_N=0$ in $\R^n\setminus B_1$ and $P_N$ is non-positive in $\R^n$, we see that
\begin{equation}\label{eq3E}
\inf_{\R^n\setminus B_1}u_N\geq0.
\end{equation}
Noting $\tau\in(0,3/4)$, we find a bound of $-\inf_{B_1}u_N$ from below as follows:
\begin{align}
-\inf_{B_1}u_N\geq -u_N(0)
&=\A(n,-(2-\sigma))\int_{\R^n}(P_N(y)-P_N(0))|y|^{-n-(2-\sigma)}dsdy\nonumber\\
&=\A(n,-(2-\sigma))\int_{\R^n}\int_0^1DP_N(sy)\cdot y|y|^{-n-(2-\sigma)}dsdy\nonumber\\
&=\A(n,-(2-\sigma))\int_{\R^n}\int_0^1\phi_N'(s|y|)|y|^{-n+\sigma-1}dsdy\nonumber\\
&\geq\A(n,-(2-\sigma))\int_{\R^n\setminus B_{1/2}}\int_{1/2}^1\phi_N'(s|y|)|y|^{-n+\sigma-1}dsdy\nonumber\\
&\geq\A(n,-(2-\sigma))\int_{\R^n\setminus B_{1/2}}\int_{1/2}^1\log(N/4)\psi(s|y|)|y|^{-n+\sigma-1}dsdy,\nonumber
\end{align}
where we have set
\[
\psi(r):=\min\left\{\frac{1-r}{\tau^2(1-\tau)^{(1-\tau)/\tau}},\frac{1}{\tau r^{(1-\tau)/\tau}}\right\}\quad\mbox{for } 0<r\leq1
\]
and $\psi=0$ for $r>1$.
Since $\psi(r)$ is decreasing,
\begin{align}
-\inf_{B_1}u_N
&\geq\A(n,-(2-\sigma))\int_{\R^n\setminus B_{1/2}}\int_{1/2}^1\log(N/4)\psi(|y|)|y|^{-n+\sigma-1}dsdy\nonumber\\
&=\A(n,-(2-\sigma))\frac{\log (N/4)}{2}\int_{B_1\setminus B_{1/2}}\psi(|y|)|y|^{-n+\sigma-1}dsdy\nonumber\\
&= c\log N/4. \label{eq4E}
\end{align}
On the other hand, Proposition \ref{Hessian} shows
\[
\M^-u_N(x)=\inf_{A\in S_{\lambda,\Lambda}}\Tr(A_\sigma D^2P_N(x)).
\]
For $1-\tau< |x|<1$, we have $\tau\phi_N''+(1-\tau)\phi_N'/|x|\leq0$, which together with \eqref{eq2E} implies
\[
\M^-u_N(x)=\inf_{A\in S_{\lambda,\Lambda}}\Tr(A_\sigma D^2P_N(x))\leq0.
\]
Similarly, for $0<|x|<1-\tau$,
\[
\M^-u_N(x)=\inf_{A\in S_{\lambda,\Lambda}}\Tr(A_\sigma D^2P_N(x))\leq\frac{\lambda(n+\sigma)}{|x|^{1/\tau}}\chi_{\{1/N<|x|<1-\tau\}}(x).
\]
Recall $p_0=n\tau$. Then, the above inequalities imply
\begin{equation}\label{eq5E}
\|(\M^-u_N)^+\|_{L^{p_0}(B_1)}\leq\lambda(n+\sigma)\left(|\partial B_1|\log(1-\tau)N\right)^{1/p_0}.
\end{equation}
Combine \eqref{eq3E}, \eqref{eq4E} and \eqref{eq5E}, to see
\[
\frac{-\inf_{B_1}u_N+\inf_{\R^n\setminus B_1}u_N}{\|(\M^-u_N)^+\|_{L^{p_0}(B_1)}}
\geq\frac{c\log N/4}{\lambda(n+\sigma)\left(|\partial B_1|\log(1-\tau)N\right)^{1/p_0}}
\to\infty\quad\mbox{as }N\to\infty.
\]
Hence \eqref{eq1E} fails at $p_0$.

In addition to the above, we can observe that the $W^{\sigma,p_0}$~estimate of form
\begin{equation}\label{eq6E}
\|D^\sigma u\|_{L^{p_0}(B_{1/2})}\leq C(n,\sigma)(\|\M^-u\|_{L^{p_0}(B_1)}+\|u\|_{L^\infty(\R^n)})
\end{equation}
also fails.
To see this assertion, it suffice to prove
\begin{align}
\|\M^-u_N\|_{L^{p_0}(B_1)}+\|u_N\|_{L^\infty(\R^n)}&\leq C\log (1-\tau)N\quad\mbox{and}\label{eq7E}\\
\|D^\sigma u_N\|_{L^{p_0}(B_{1/2})}&\geq c(\log (1-\tau)N)^{1+1/p_0}\label{eq8E}
\end{align}
for some $c,C>0$.
In this way, we will show \eqref{eq6E} fails from
\[
\frac{\|D^\sigma u_N\|_{L^{p_0}(B_{1/2})}}{\|\M^-u_N\|_{L^{p_0}(B_1)}+\|u_N\|_{L^\infty(\R^n)}}\geq c(\log (1-\tau)N)^{1/p_0}\to\infty\quad\mbox{as }N\to\infty.
\]

First, we confirm \eqref{eq7E} holds.
Note that we have $\M^-u_N(x)>0$ for $x\in B_{1-\tau}$ by applying \eqref{eq0E} and
\[
0<\tau\phi''_N(r)+(1-\tau)\phi'_N(r)\leq\frac{1}{\sigma+n}\phi_N''(r)+\frac{\sigma+n-1}{\sigma+n}\phi_N'(r) \quad\mbox{for }0<r\leq1-\tau.
\]
Hence, together with \eqref{eq5E}, we have
\[
\|\M^-u_N\|_{L^{p_0}(B_{1-\tau})}=\|(\M^-u_N)^+\|_{L^{p_0}(B_{1-\tau})}\leq\lambda(n+\sigma)\left(|\partial B_1|\log(1-\tau)N\right)^{1/p_0}.
\]
On the other hand, for $1-\tau<r\leq1$, we have
\[
|\phi_N'|+|\phi_N''|\leq\frac{2\log(1-\tau)N}{\tau^2(1-\tau)^{(1-\tau)/\tau}}
\]
Hence it follows that
\begin{align*}
\|\M^-u_N\|_{L^{p_0}(B_1\setminus B_{1-\tau})}
&\leq C(\|\phi_N'\|_{L^\infty(B_{B_1\setminus B_{1-\tau}})}+\|\phi_N''\|_{L^\infty(B_{B_1\setminus B_{1-\tau}})})\\
&\leq C\log(1-\tau)N.
\end{align*}
and combining the above two inequalities, we have $\|\M^-u_N\|_{L^{p_0}(B_{1})}\leq C\log(1-\tau)N$.
We next show $\|u_N\|_{L^\infty(\R^n)}\leq C\log (1-\tau)N$.
Since
\[
|DP_N(x)|=|\phi'_N(|x|)|\leq\frac{\log(1-\tau)N}{\tau}\max\left\{\frac{1}{(1-\tau)^{(1-\tau)/\tau}},\frac{1}{|x|^{(1-\tau)/\tau}}\right\}\chi_{B_1}(x)\quad\mbox{for }x\in\R^n,
\]
we have $\|DP_N\|_{L^{p\tau/(1-\tau)}(\R^n)}\leq C\log(1-\tau)N$ for $0<p<n$.
By Morrey's inequality under $n<p\tau/(1-\tau)$, we also obtain $\|P_N\|_{C^{\gamma}(\R^n)}\leq C\log(1-\tau)N$ for $0<\gamma<(2\tau-1)/\tau$.
Note that we can choose $\gamma$ such that $2-\sigma<\gamma<(2\tau-1)/\tau$ for $\sigma\in(\sqrt{n},2)$.
Then, we have
\begin{align*}
|u_N(x)|
&=|\A(n,-(2-\sigma))\int_{\R^n}\left(P_N(x)-P_N(x+y)\right)\frac{dy}{|y|^{n+(2-\sigma)}}|\\
&\leq \A(n,-(2-\sigma))\|P_N\|_{C^{\gamma}(\R^n)}\int_{\R^n}\min\{|y|^{\gamma},1\}\frac{dy}{|y|^{n+(2-\sigma)}}\\
&\leq C\log(1-\tau)N,
\end{align*}
where we have used the fact that $\min\{|y|^{\gamma},1\}/|y|^{n+(2-\sigma)}$ is integrable for $\gamma>2-\sigma$.
Hence \eqref{eq7E} is proved.

Now. it remains to prove \eqref{eq8E}.
It follows that
\begin{align*}
\|D^\sigma u_N\|_{L^{p_0}(B_{1/2})}=\|[D^2P_N]_\sigma\|_{L^{p_0}(B_{1/2})}\geq c\|D^2 P_N\|_{L^{p_0}(B_{1/2})}&\geq c\|\phi_N''(|\cdot|)\|_{L^{p_0}(B_{1/2})}.
\end{align*}
Note that $\tau >1/2$ since $n=2,3$ and $\sigma\in (\sqrt{n},2)$.
Then, We calculate as follow:
\begin{align*}
\int_{B_{1/2}}|\phi_N''(|x|)|^{p_0}dx
&\geq\int_{B_{1-\tau}}|\phi_N''(|x|)|^{p_0}dx\\
&=\frac{c}{\tau^{p_0}}\int_{1/N}^{1-\tau}r^{-1}\left|\frac{1-\tau}{\tau}\log rN-1\right|^{p_0}dr\\
&\geq \frac{c}{\tau^{p_0}}\int_{1/N}^{1-\tau}r^{-1}\left(\frac{1-\tau}{\tau}\log rN-1\right)^{p_0}dr\\
&= \frac{c}{\tau^{p_0-1}(p_0+1)(1-\tau)}\left[\left(\frac{1-\tau}{\tau}\log rN-1\right)^{p_0+1}\right]_{1/N}^{1-\tau}\\
&= \frac{c}{\tau^{p_0-1}(p_0+1)(1-\tau)}\left(\left(\frac{1-\tau}{\tau}\log (1-\tau)N-1\right)^{p_0+1}-(-1)^{p_0+1}\right),
\end{align*}
where $c$ is a constant depending only on $n$. 
Hence by choosing $N$ large enough, we arrive at \eqref{eq8E}.

\section{Appendix. The proof of Proposition \ref{infconv}}
We first recall some properties of semiconcave functions (see Lemma 3.3 and Lemma 3.15 in~\cite{Jensen}).
\begin{prop}\label{Prop61}
Let $u \in LSC(\R^n)$ and $u_h$ be the inf-convolution of $u$, defined by \eqref{uh}.
Then there exists a function $M\in L^1_{loc}(\R^n,\mathscr{S}^n)$ and a matrix-valued measure $S\in\mathscr{M}(\R^n;\mathscr{S}^n)$ such that
\begin{enumerate}
\item[(a)] $D^2u_h=M+S$ in the sense of distributions.
\item[(b)] $S$ is singular with respect to Lebesgue measure.
\item[(c)] $M(x)\leq h^{-1}Id$ for a.e. $x\in\R^n$.
\item[(d)] for a.e. $x\in\R^n$,
\[
u_h(x+y)=u_h(x)+Du_h(x)\cdot y+\frac{1}{2}\langle M(x)y,y\rangle+o(|y|^2).
\]
\end{enumerate}
Here we have denoted by $\mathscr{S}^n$, the set of real symmetric $n\times n$ matrices.
\end{prop}

\noindent{\bf Proof of Proposition \ref{infconv}.}
(i) of Proposition \ref{infconv} immediately follows from (d) of Proposition \ref{Prop61}.
In fact, for any $x\in\R^n$ such that $|\delta(u_h,x,y)|\leq |M(x)||y|^2+o(|y|^2)$ holds, $\delta(u_h,x,y)/|y|^{n+\sigma}$ is integrable, and so $Iu_h(x)$ is evaluated classically.

We next prove (ii). Suppose $\phi$ touches $u_h$ from below at $x_0$ in some neighborhood $N$. For any $h>0$, set
\[
v_{h}:=
\begin{cases}
\phi&\mbox{in }\overline{N},\\
u_h&\mbox{in }\R^n\setminus N.
\end{cases}
\]
Then, we have
\[
v_{h}(x)\leq u_h(x)\quad\mbox{for }x\in\R^n\quad\mbox{and}\quad v_{h}(x_0)=u_h(x_0).
\]
Note that $u\in LSC(\R^n)\cap L^\infty(\R^n)$, and hence there exists $x^*\in\R^n$ such that
\[
u_h(x_0)=\inf_{x\in\R^n}\left\{u(x)+\frac{1}{2h}|x-x_0|^2\right\}=u(x^*)+\frac{1}{2h}|x^*-x_0|^2,
\]
since $u(x)+|x-x_0|^2/2\to\infty$ as $|x|\to\infty$.
Hence for $x,y\in\R^n$, we have
\[
v_{h}(x)\leq u(y)+\frac{1}{2h}|y-x|^2\quad\mbox{and}\quad v_{h}(x_0)=u(x^*)+\frac{1}{2h}|x^*-x_0|^2.
\]
Fix $x=y+x_0-x^*$, to see
\[
v_h(y+x_0-x^*)-\frac{1}{2h}|x^*-x_0|^2\leq u(y)\quad\mbox{and}\quad v_h(x_0)-\frac{1}{2h}|x^*-x_0|^2=u(x^*),
\]
Hence $v_{h}(y+x_0-x^*)-(1/2h)|x^*-x_0|^2$ touches $u(y)$ from below at $x^*$, which together with the assumption that $u$ is a viscosity supersolution of $Iu=f$ in $\R^n$ implies
\begin{equation}\label{ineq61}
Iv_{h}(x_0)=I[v_{h}(\cdot+x_0-x^*)](x^*)\leq f(x^*).
\end{equation}
By the definition of $x^*$, we have
\begin{align*}
&\quad u(x^*)+\frac{1}{2h}|x^*-x_0|^2\leq u(x)+\frac{1}{2h}|x-x_0|^2\\
&\Leftrightarrow |x^*-x_0|^2\leq 2h\left\{u(x)-u(x^*)+\frac{1}{2h}|x-x_0|^2\right\}
\end{align*}
for $x\in\R^n$. Choosing $x=x_0$, we see
\begin{equation*}
|x^*-x_0|^2\leq 2h\{u(x_0)-u(x^*)\}\leq 4h\|u\|_\infty.
\end{equation*}
Hence from this inequality and \eqref{ineq61}, (ii) is proved.

Finally, we check (iii). Obviously, $Iu_{h,\e}(x)\leq C\max\{h^{-1},\|u\|_\infty\}$ follows from semiconcavity of $u_h$.
We shall show $D^\sigma u_{h,\e}(x)\to D^\sigma u_h(x)$, which implies $Iu_{h,\e}(x)\to Iu_h(x)$ for a.e. in $\R^n$ as $\e\to0$.
From (a) of Proposition \ref{Prop61}, we may fix $x\in\R^n$ such that
\begin{equation*}
\dashint_{B_l}|D^2u_h(x+y)|dy\to |M(x)|\quad\mbox{as }l\to0,
\end{equation*}
and so
\begin{equation*}
\sup_{0<l<2}\dashint_{B_l}|D^2u_h(x+y)|dy<\infty,
\end{equation*}
where we have used the notation: $\dashint_{B_s}f(y)dy:=|B_s|^{-1}\int_{B_s}f(y)dy$. By Taylor's theorem,
\begin{align*}
\delta(u_{h,\e},x,y)
&=u(x+y)+u(x-y)-2u(x)\\
&=\int_{-1}^1(1-\tau)\langle D^2u_{h,\e}(x+\tau y)y,y\rangle d\tau\\
&=\int_{-1}^1\int_{B_\e}(1-\tau)\langle D^2u_h(x+\tau y-z)y,y\rangle\e^{-n}\psi(\e^{-1}z)dzd\tau.
\end{align*}
For any $s\in(0,1)$, we calculate as follows:
\begin{align*}
&\quad\left|\int_{B_s\setminus B_{s/2}}\delta(u_{h,\e},x,y)\frac{y\otimes y}{|y|^{n+\sigma+2}}dy\right|\\
&\leq C\e^{-n}s^{-n+(2-\sigma)}\int_{-1}^1\int_{B_s}\int_{B_\e}|D^2 u_h(x-z+\tau y)|dzdyd\tau\\
&=Cs^{2-\sigma}\int_{-1}^1\int_{B_{\tau s}}\int_{B_\e}(\e\tau s)^{-n}|D^2 u_h(x-z+y)|dzdyd\tau\\
&\leq Cs^{2-\sigma}\int_{-1}^1\dashint_{B_{\tau s+\e}}|D^2u_h(x+y)|dyd\tau\leq Cs^{2-\sigma}\sup_{0<l<2}\dashint_{B_l}|D^2u_h(x+y)|dy
\end{align*}
Hence it follows that for any $r>0$,
\begin{align*}
\left|\int_{B_r}\delta(u_{h,\e},x,y)\frac{y\otimes y}{|y|^{n+\sigma+2}}dy\right|&\leq\sum_{k=0}^\infty\left|\int_{B_{r2^{-k}}\setminus B_{r2^{-k-1}}}\delta(u_{h,\e},x,y)\frac{y\otimes y}{|y|^{n+\sigma+2}}dy\right|\\
&\leq \sum_{k=0}^\infty Cr^{2-\sigma}2^{-k(2-\sigma)}\sup_{0<l<2}\dashint_{B_{l}}|D^2u_h(x+y)|dyd\tau\\
&=O(r^{2-\sigma}).
\end{align*}
Consequently,
\begin{align*}
D^\sigma u_{h,\e}(x)
&=\frac{\A(n,-\sigma)}{2}\left(\int_{\R^n\setminus B_r}\delta(u_{h,\e},x,y)\frac{y\otimes y}{|y|^{n+\sigma+2}}dy+\int_{B_r}\delta(u_{h,\e},x,y)\frac{y\otimes y}{|y|^{n+\sigma+2}}dy\right)\\
&\to\frac{\A(n,-\sigma)}{2}\int_{\R^n\setminus B_r}\delta(u_{h},x,y)\frac{y\otimes y}{|y|^{n+\sigma+2}}dy+O(r^{2-\sigma})\quad\mbox{as }\e\to0.
\end{align*}
Then, by letting $r\to0$, we find $D^\sigma u_{h,\e}(x)\to D^\sigma u_h(x)$ as $\e\to0$.$\quad\Box$

\noindent{\bf Acknowledgments.} I am grateful to Prof. Shigeaki Koike for encouragement and his careful reading of the manuscript and also the anonymous referees for so thoroughly reading the manuscript and for giving constructive comments.
S. Kitano is supported by Grant-in-Aid for JSPS Fellows 21J10020.

\end{document}